\documentclass[a4paper,11pt]{amsart}
\usepackage{amsmath}
\usepackage{amsthm}
\usepackage{amssymb}
\usepackage{enumerate}
\usepackage{enumitem}
\usepackage{mathrsfs}
\DeclareMathOperator{\Stab}{Stab}
\DeclareMathOperator{\FC}{FC}
\usepackage{showlabels}
\usepackage{hyperref}
\usepackage{xcolor}
\setlist[itemize]{leftmargin=15pt}

\theoremstyle{plain}
\newtheorem{thm}{Theorem}[section]
\newtheorem{cor}[thm]{Corollary}
\newtheorem{lem}[thm]{Lemma}
\newtheorem{prop}[thm]{Proposition}

\theoremstyle{definition}

\newtheorem{eg}[thm]{Example}
\newtheorem{rem}[thm]{Remark}
\newtheorem{con}[thm]{Construction}
\newtheorem{op}[thm]{Open Problem}

\title{On separability finiteness conditions in semigroups}
\author{Craig Miller, Gerard O'Reilly, Martyn Quick, Nik Ru{\v s}kuc}
\address{School of Mathematics and Statistics, University of St Andrews, St Andrews, Scotland, UK}
\email{$\{$cm380, gao2, mq3, nr1$\}$@st-andrews.ac.uk}

\keywords{Separability, finiteness condition, semigroup, congruence, residual finiteness, commutative semigroup, Sch\"utzenberger group}
\date{}

\begin{document}

\begin{abstract}
	Taking residual finiteness as a starting point, we consider three related finiteness properties: weak subsemigroup separability, strong subsemigroup separability and complete separability.
	We investigate whether each of these properties is inherited by Sch\"utzenberger groups.
	The main result of this paper states that for a finitely generated commutative semigroup $S$, these three separability conditions coincide and are equivalent to every $\mathcal{H}$-class of $S$ being finite. 
	We also provide examples to show that these properties in general differ for commutative semigroups and finitely generated semigroups. 
	For a semigroup with finitely many $\mathcal{H}$-classes, we investigate whether it has one of these properties if and only if all its Sch\"utzenberger groups have the property.
\end{abstract}

\maketitle

\section{Introduction}
A finiteness condition for a class of algebraic structures is a property that is satisfied by at least all finite members of that class. 
The study of finiteness conditions has been instrumental in understanding the structure and behaviour of algebras. 
A classic example is that of residual finiteness. 
An algebraic structure $A$ is said to be \emph{residually finite} if for any two distinct points $x,y \in  A$ there exists a finite algebra $U$ and a homomorphism $f: A \to U$ such that $f(x) \neq f(y)$. 
The property of residual finiteness is well studied and has proved to be a powerful tool. 
For example, it is known that if an algebra is finitely presented and residually finite then its word problem is solvable, as shown by Evans \cite{evans1978} who attributes this to Mal'cev. 
Another important example where this property arises is in the context of Zelmanov's positive solution \cite{Z1, Z2} of the Restricted Burnside Problem, which can be interpreted as saying that a finitely generated, residually finite group of finite exponent is necessarily finite.

Residual finiteness is an instance of what we call a \emph{separability finiteness condition}.
The notion of separability concerns separating an element from a subset in a finite quotient.
Formally, given a class of algebras $\mathcal{C}$, an algebra $A\in \mathcal{C}$, an element $x \in A$ and a subset $S \subseteq A\!\setminus\! \{x\}$, we say that $x$ can be \emph{separated} from $S$ if there exists a finite algebra $U\in\mathcal{C}$ and a homomorphism $f:A \to U$ such that $f(x) \notin f(S)$. 
In this case we say that $f$ \emph{separates} $x$ from $S$. 
Considering collections of subsets of a certain type (such as finite subsets, subalgebras, etc.)\ gives rise to various separability finiteness conditions. 
More precisely, let $\mathcal{C}$ be a class of algebras, let $A \in \mathcal{C}$ and $\mathcal{S}$ a collection of subsets of $A$. 
We say that an algebra $A$ has the \emph{separability property with respect to $\mathcal{S}$} if for any $x \in A$ and any $\mathcal{S}$-subset $Y \subseteq A\!\setminus\!\{x\}$, the element $x$ can be separated from $Y$.
We note that when $\mathcal{S}$ is the collection of singleton subsets, this is equivalent to residual finiteness. 
If the class $\mathcal{C}$ of algebras is closed under direct products, then residual finiteness is equivalent to an algebra having the separability property with respect to the collection of finite subsets.
By varying $\mathcal{S}$, we will consider several properties. 
Many of these properties have already been studied in different contexts and under various names.

When $\mathcal{S}$ is the collection of all finitely generated subalgebras, we say that an algebra is \emph{weakly subalgebra separable.} 
Weak subalgebra separability can play a similar role to residual finiteness, in that if an algebra is finitely presented and weakly subalgebra separable then the generalised word problem is solvable \cite{evans1978}. 
Evans referred to weak subalgebra separability as \emph{finite divisibility}. 
In group theory, this property is known simply as \emph{subgroup separability} or \emph{locally extended residual finiteness} (LERF)\@. 
This group-theoretic property has received considerable attention. Many classes of groups have been shown to be weakly subgroup separable including free groups \cite{10.2307/1990483}, fundamental groups of geometric 3-manifolds \cite{MR3104553}, and finitely generated nilpotent groups (which includes finitely generated abelian groups)  \cite{Malcev}.
Within semigroup theory this property has received less attention, but Golubov did briefly consider it amongst some other separability properties in \cite{MR0283113}.

If $\mathcal{S}$ is the collection of all subalgebras, we say that an algebra is \emph{strongly subalgebra separable}. 
In group theory this is known as \emph{extended residual finiteness} (ERF)\@. 
Although strongly subgroup separable groups have not received as much attention as weakly subgroup separable groups, they have been studied to some extent, for example in \cite{ROBINSON2009421}. 
In semigroup theory, this property has been known as \emph{finite divisibility} or \emph{finite separability}. 
Strongly subsemigroup separable semigroups were considered by Lesohin in \cite{MR0238969, MR0244416} and were then intensively studied by Golubov in \cite{MR0274613, MR0320186, MR318355, MR1094802}. 
In \cite{MR0274613} Golubov characterised when commutative semigroups are strongly subsemigroup separable.

Finally, we consider the case where $\mathcal{S}$ is the collection of all subsets. 
Then we are able to separate any point in an algebra from any subset of its complement. 
We call algebras satisfying this condition \emph{completely separable}. 
Completely separable semigroups are also known as \emph{semigroups with finitely divisible subsets} and again were studied by Golubov. 
In \cite{MR0274613} he showed that free semigroups and free commutative semigroups are completely separable, and also proved that complete separability and strong subsemigroup separability coincide for semigroups without idempotents. 

Due to the many different and sometimes inconsistent names used for these properties, we have decided to introduce our own terms, i.e., weak subalgebra separability, strong subalgebra separability and complete separability. 
In our nomenclature, the names are designed to describe the properties and highlight the relationship between the different properties. 
Under this system of names, it might be more appropriate to call residual finiteness \emph{point separability}. 
However, as the term `residually finite' is now universally used in literature, we have decided to maintain its use. 

The purpose of this paper is to discuss the properties of weak subalgebra separability, strong subalgebra separability and complete separability for semigroups, and to explore the relationships between them. 
In Section 2, we outline the necessary preliminary definitions and results. 
In Section 3, we investigate how Sch\"utzenberger groups are affected by these properties. 
This builds upon observations of when groups have the different separability properties within the class of groups and when they have them within the class of semigroups. 
Somewhat surprisingly, given that Sch\"utzenberger groups can be seen as a generalisation of maximal subgroups, we discover that Sch\"utzenberger groups of non-regular $\mathcal{H}$-classes exhibit different behaviour from maximal subgroups. 

Section 4 begins with a brief summary of work by Kublanovski{\u \i} and Lesohin \cite{MR541122}. 
These authors showed that in finitely generated commutative semigroups, the properties of complete separability and strong subsemigroup separability coincide. 
We strengthen their result by showing that in finitely generated commutative semigroups, complete separability also coincides with weak subsemigroup separability, and they are all equivalent to every $\mathcal{H}$-class of $S$ being finite. 

The question arises whether these three properties coincide in more general classes of semigroups. 
Section 5 is dedicated to providing examples to show that all three are in fact different for semigroups in general, and in particular for (infinitely generated) commutative semigroups and for finitely generated (non-commutative) semigroups.

Finally, in Section 6, we consider a special case where a semigroup only has finitely many $\mathcal{H}$-classes. 
For a semigroup of this type, we investigate whether it has one of our separability conditions if and only if all of its Sch\"utzenberger groups do. 
We solve this question for complete separability and strong subsemigroup separability. 
For weak subsemigroup separability this remains an open question, but we provide a partial solution.

\section{Preliminaries}
We begin by establishing some notation. 
For a semigroup $S$ and a non-empty subset $X \subseteq S$, we use $\langle X \rangle$ to denote the subsemigroup generated by the set $X$. 
For a congruence $\sim$ on $S$, we denote the congruence class of the element $x$ by $[x]_\sim$.
A congruence that has only finitely many congruence classes shall be known as a \emph{finite index congruence}.
Given an ideal $I$ of a semigroup $S$, the \emph{Rees congruence} $\sim_I$ on $S$ is given by 
$$a\,\sim_I\,b\iff a, b\in I\text{ or }a=b.$$  
The quotient $S/\!\!\sim_I$ is called the \emph{Rees quotient} of $S$ by $I$, and is denoted by $S/I$.  
We denote the Rees congruence class of an element $x\in S$ by $[x]_I$.

In the introduction we framed separability in the language of general algebra. 
From now on we focus on separability for \emph{semigroups}. 
For clarity we now formally restate our separability conditions in the semigroup setting and also reformulate them in terms of congruences.  
\begin{itemize}
	\item A semigroup $S$ is said to be \emph{residually finite} if for any two distinct elements $s,t  \in S$ there exists a finite semigroup $U$ and homomorphism $f: S \to U$ such that $f(s) \neq f(t)$. 
	This is equivalent to saying there exists a congruence $\sim$ of finite index such that $[s]_{\sim}\neq [t]_{\sim}$.  
	
	\item A semigroup $S$ is \emph{weakly subsemigroup separable} if for any finitely generated subsemigroup $T$ and any $s \in S\!\setminus\!T$ there exists a finite semigroup $U$ and homomorphism $f: S \to U$ such that $f(s) \notin f(T)$. 
	This is equivalent to saying there exists a congruence $\sim$ of finite index such that $[s]_{\sim} \neq [t]_ {\sim}$ for all $ t \in T$. 
	
	\item A semigroup $S$ is \emph{strongly subsemigroup separable} if for any subsemigroup $T$ and any $s \in S \!\setminus\! T$ there exists a finite semigroup $U$ and homomorphism $f: S \to U$ such that $f(s) \notin f(T)$. 
	Again, this is equivalent to saying there exists a congruence $\sim$ of finite index such that $[s]_{\sim} \neq [t]_ {\sim}$ for all $ t \in T$.  
	
	\item A semigroup $S$ is \emph{completely separable} if for any $s \in S$ there exists a finite semigroup $U$ and homomorphism $f:S \to U$ such that $f(s) \notin f(S\!\setminus\! \{s\})$. 
	This is equivalent to saying there exists a congruence $\sim$ of finite index such that $[s]_{\sim}=\{s\}$.
\end{itemize}
For a semigroup $S$, the \emph{profinite topology} on $S$ is defined as the topology on $S$ with a basis consisting of all congruence classes of all finite index congruences on $S$. 
Any semigroup equipped with the profinite topology is a topological semigroup. 
We can now express the above properties in topological terms.

\begin{itemize}
\item A semigroup is residually finite if and only if every singleton is closed in the profinite topology. This is equivalent to saying that the intersection of all finite index congruences is the equality relation. 
\item A semigroup is weakly subsemigroup separable if and only if all finitely generated subsemigroups are closed in the profinite topology. 
\item A semigroup is strongly subsemigroup separable if and only if all subsemigroups are closed in the profinite topology. 
\item A semigroup is completely separable if and only if the profinite topology is discrete.
\end{itemize}

We state how these properties relate to one another in the following result. 

\begin{prop}
	\label{lem}
	For a semigroup $S$ the following hold:
	\begin{enumerate}
		\item[\textup{(1)}] If $S$ is completely separable then it is strongly subsemigroup separable.
		\item[\textup{(2)}] If $S$ is strongly subsemigroup separable then it is weakly subsemigroup separable.
		\item[\textup{(3)}] If $S$ is weakly subsemigroup separable then it is residually finite.
	\end{enumerate}
\end{prop}

\begin{proof}
	It is clear from the definitions that the first two claims are true. 
	For (3), assume that $S$ is weakly subsemigroup separable and let $a,b$ be distinct elements of $S$. 
	Let $A=\langle a \rangle$ and $B=\langle b \rangle$. We split into two cases. 
	Firstly we assume that one of $a,b$ is not contained in the monogenic subsemigroup generated by the other. 
	Without loss of generality assume that $a \notin B$. 
	Then as $S$ is subsemigroup separable, there exists a finite semigroup $U$ and homomorphism $f:S \to U$ such that $f(a) \notin f(B)$. 
	As $f(b) \in f(B)$ it follows that $f(a) \neq f(b)$.
	
	The second case is that $a \in B$ and $b \in A$. 
	Then $a=b^i$ for some $i \in \mathbb{N}$ and $b=a^j$ for some $j \in \mathbb{N}$. 
	Then it must be the case that $A=B$ and that $A$ is a finite cyclic group. Let $e$ denote the identity of $A$ and let $c$ be the inverse of $b$. 
	Then $\langle e \rangle =\{e\}$. 
	We note that neither $a$ nor $b$ is equal to $e$, as in either case this would imply that $a=b=e$.
	As $a \neq b$ it follows that $ac \neq e$. 
	Then, as $S$ is weakly subsemigroup separable, there exists a finite semigroup $U$ and homomorphism $f: S \to U$ such that $f(ac) \neq f(e)$. 
	It follows that $f(a) \neq f(b)$.
	Hence $S$ is residually finite.
\end{proof}

\begin{rem}
	\normalfont It is clear that the first two statements of Proposition \ref{lem} hold in any class of algebras. 
	However, it is not true in general that weak subalgebra separability implies residual finiteness.
	 One example is the class of fields. 
	 Since $\mathbb{Q}$ has no proper subfields, it is weakly subfield separable. 
	 However, $\mathbb{Q}$ has no finite quotients so it is not residually finite.	
\end{rem}

In general, the converse of each implication from Proposition \ref{lem} is not true. 
This is demonstrated in Examples \ref{eg1}, \ref{eg5} and \ref{square-free}. 
For every finite semigroup the profinite topology is discrete, so we can immediately observe that every finite semigroup is completely separable and hence that all four separability properties are finiteness conditions. 
The following proposition shows that subsemigroups inherit each of these separability properties.

\begin{prop}
	\label{lem5}
	Let $S$ be a semigroup and $T$ a subsemigroup of $S$. Let $\mathscr{P}$ be any of the following properties: complete separability, strong subsemigroup separability, weak subsemigroup separability, and residual finiteness. 
	If $S$ has property $\mathscr{P}$ then $T$ also has property $\mathscr{P}$.
\end{prop}

\begin{proof}
	As $S$ has property $\mathscr{P}$, it has the separability property with respect to $\mathcal{C}$, where $\mathcal{C}$ is the collection of subsets of the type associated with $\mathscr{P}$. 
	Let $X \subseteq T$ be a subset of the relevant type and let $t \in T \! \setminus \! X$. 
	Then $X$ is also a subset of the relevant type in $S$ and $ t \in S \! \setminus \! X$. 
	Then as $S$ has property $\mathscr{P}$, there exists a finite semigroup $U$ and homomorphism $f:S \to U$ such that $f(t) \notin f(S \!\setminus\! X)$. 
	Let $f|_T: T \to U$ be the restriction of $f$ to $T$. Then $f|_T(t) \notin f|_T(T \!\setminus\! X)$ and $T$ has property $\mathscr{P}$.
\end{proof}

In Section 3 we consider when certain groups have our separability properties. 
We consider these groups as semigroups so it is worth noting how the separability conditions vary between the class of groups and the class of semigroups. 

A group $G$ is residually finite within the class of groups if and only if it is residually finite within the class of semigroups. 
This is because every homomorphic image of a group is itself a group. 
Similarly, we can see that a group $G$ is completely separable within the class of groups if and only if it is completely separable within the class of semigroups. 
Furthermore, we can fully classify completely separable groups. 

\begin{lem}
	\label{comp.sep.group}
	A group $G$ is completely separable if and only if it is finite.
\end{lem}

\begin{proof}
	We have already observed that all finite groups are completely separable. 
	Now suppose that $G$ is completely separable. 
	Then there exists a finite index congruence on $G$ such that $\{e\}$ is a congruence class, where $e$ is the identity element. 
	Then $[G:\{e\}]$ is finite, implying that $G$ is finite.
\end{proof}	

We note that there exist infinite groups which are strongly subsemigroup separable (and hence weakly subsemigroup separable), as Theorem \ref{thm3} demonstrates.

It is also clear that if a group $G$ is strongly subsemigroup separable, then it must also be strongly subgroup separable. 
This is because every subgroup is also a subsemigroup. 
However, if a group is strongly subgroup separable then it may not be strongly subsemigroup separable (or even weakly subsemigroup separable), as the following example demonstrates.
\begin{eg}
	\label{eg1}
	\normalfont Let $\mathbb{Z}$ denote the group of integers under addition. 
	As every non-trivial subgroup of $\mathbb{Z}$ is a normal subgroup of finite index and their intersection is $\{0\}$, it follows that $\mathbb{Z}$ is strongly subgroup separable. 
	Now consider the subsemigroup $\mathbb{N}=\langle 1 \rangle \leq \mathbb{Z}$ and the element $-1 \notin \mathbb{N}$. 
	Let $U$ be a finite semigroup and $\phi: \mathbb{Z} \to U$ be a homomorphism. 
	Then $\phi(\mathbb{Z})$ is a finite cyclic group, generated by $\phi(1)$. 
	Hence $\phi(-1) \in \langle \phi(1) \rangle = \phi(\mathbb{N})$. 
	Therefore $\mathbb{Z}$ is not even weakly subsemigroup separable. 
	Since $\mathbb{Z}$ is residually finite, this shows that the properties of residual finiteness and weak subsemigroup separability are distinct in semigroups, and in particular in finitely generated commutative semigroups. 
\end{eg}

As a consequence of Proposition \ref{lem5} and Example \ref{eg1}, we see that if a group $G$ is weakly subsemigroup separable then it cannot contain a copy of $\mathbb{Z}$. 
In other words, for $G$ to be weakly subsemigroup separable it is necessary for it to be torsion. 
As every subsemigroup of a torsion group is in fact a subgroup, we observe the following.
\begin{prop}
\label{GroupSep}
\leavevmode 
\begin{enumerate}[noitemsep, nolistsep]
	\item[\textup{(1)}] A group $G$ is weakly subsemigroup separable if and only if $G$ is torsion and weakly subgroup separable.
	\item[\textup{(2)}] A group $G$ is strongly subsemigroup separable if and only if $G$ is torsion and strongly subgroup separable.
\end{enumerate}
\end{prop}

We will now briefly discuss the situation for abelian groups. 
It is known that a group is residually finite if and only if it is isomorphic to a subdirect product of finite groups, see \cite[Corollary 7.2]{cohn1981universal}. 
It follows that an abelian group is residually finite if and only if it is isomorphic to a subdirect product of finite cyclic groups.

In order to discuss weak subsemigroup separability in abelian groups, we first give the following lemma. 
A semigroup is called \emph{locally finite} if every finitely generated subsemigroup is finite. 

\begin{lem}
\label{locallyfinite}
Let $S$ be a semigroup which is both residually finite and locally finite. 
Then $S$ is weakly subsemigroup separable.
\end{lem}

\begin{proof}
Let $T \leq S$ be finitely generated and $x \notin T$. 
Then as $S$ is locally finite, $T$ is finite, say $T=\{t_1, t_2, \dots, t_n\}$. 
For each $i$, there is a finite semigroup $P_i$ and a homomorphism $\phi_i: S \to P_i$ such that $\phi_i(x) \neq \phi_i(t_i)$.
Then $\phi: S \to P_1 \times P_2 \times \dots \times P_n$ given by $s \mapsto \left(\phi_1(s), \phi_2(s), \dots, \phi_n(s)\right)$ is a homomorphism which separates $x$ from $T$.
\end{proof}

For an abelian group to be weakly subsemigroup separable, it is necessary for it to be residually finite. 
It is also necessary for it to be torsion, as noted above.
Since torsion abelian groups are locally finite, being residually finite and torsion are sufficient conditions for an abelian group to be weakly subsemigroup separable by Lemma \ref{locallyfinite}.

In \cite{MR0274613} Golubov was able to characterise when commutative semigroups are strongly subsemigroup separable. 
We can apply his result to abelian groups. 
For an abelian group $A$ and for a prime $p$, recall that the \emph{$p$-primary component} of $A$ is the set 
\[A_p=\{a \in A \mid o(a)=p^n \text{ for some } n \in \mathbb{N}\}\] 
where $o(a)$ is the order of the element $a$. 
We say that $A_p$ has \emph{finite exponent} if there exists $n \in \mathbb{N}$ such that $o(a) \leq p^n$ for all $a \in A_p$. 
Then Golubov's result, \cite[Theorem 2]{MR0274613}, tells us that an abelian group $A$ is strongly subsemigroup separable if and only if $A$ is torsion and for each prime $p$, the primary $p$-component has finite exponent. 

Finally, Lemma \ref{comp.sep.group} tells us that a group is completely separable if and only if it is finite. 
We summarise all these observations.

\begin{thm}
	\label{thm3}
	Let $A$ be an abelian group. 
	\begin{enumerate}
		\item[\textup{(1)}] $A$ is residually finite if and only if it is isomorphic to a subdirect product of finite cyclic groups.
		\item[\textup{(2)}] $A$ is weakly subsemigroup separable if and only if it is torsion and residually finite.
		\item[\textup{(3)}] $A$ is strongly subsemigroup separable if and only if it is torsion and for each prime $p$, the primary $p$-component has finite exponent.
		\item[\textup{(4)}] $A$ is completely separable if and only if it is finite. 
	\end{enumerate}
\end{thm}

In the remainder of this section we record some structural results for semigroups that will aid us in our investigations. 
For a semigroup $S$, let $S^1$ denote $S$ with an identity adjoined if necessary. 
\emph{Green's relation $\mathcal{H}$} on $S$ is defined by \[\mathcal{H}=\{(x,y) \mid xS^1=yS^1 \text{ and } S^1x=S^1y\}.\] 
The $\mathcal{H}$ relation is an equivalence relation, and in commutative semigroups it is also a congruence. 
For an element $x \in S$, we denote the $\mathcal{H}$-class of $x$ by $H_x$. 
For an $\mathcal{H}$-class $H$, the following are equivalent: $H$ is a maximal subgroup; $H$ contains an idempotent; the intersection  $H \cap H^2$ is non-empty \cite[Corollary 2.2.6]{howie1995fundamentals}.

In fact, for an $\mathcal{H}$-class $H$, if $H \cap H^n \neq \emptyset$ for any $n \geq 2$ then $H$ is a maximal subgroup. 
This is folklore but we provide a proof for completeness.  
Suppose $h_1,h_2,\dots,h_n,h \in H$ such that $h_1h_2 \dots h_n=h.$ Since $h_1,  h_2,  h_{n-1},  h_{n}$ and $h$ are pairwise $\mathcal{H}$-related there exist $s,  t,  u,  v \in S^1$ such that
\[ hs=h_1,h_nt=h_2,uh=h_n,vh_1=h_{n-1}.\]
Then
\[h\cdot sh_n=h_1h_n, \: h_1h_n \cdot th_3h_4 \dots h_n=h,\]
\[ h_1u \cdot h=h_1h_n, \: h_1h_2\dots h_{n-2}v \cdot h_1h_n=h.\]
Hence $(h,h_1h_n) \in \mathcal{H}$. 
Then $H \cap H^2 \neq \emptyset$ and it follows that $H$ is a group.

It is possible to associate a group to an arbitrary $\mathcal{H}$-class $H$ as follows. 
The \emph{right stabiliser} of $H$ in $S$ is 
\[\Stab(H)=\{s \in S^1 \mid Hs=H\}.\] 
Clearly $\Stab(H)$ is a submonoid of $S^1$. 
We define a congruence $\sigma_H$ on $\Stab(H)$, called the \emph{Sch\"utzenberger congruence of $H$}, by 
\[(x,y) \in \sigma_H \iff hx=hy \text{ for all }h \in H.\] 
Then $\Gamma(H)=\Stab(H)/\sigma_H$ is a group, known as the \emph{Sch\"utzenberger group} of $H$. 
It is known that:
\begin{itemize}
	\item $\Gamma(H)$ acts regularly, i.e., transitively and freely, on $H$;
	\item $|\Gamma(H)|=|H|$;
	\item if $H$ is a group then $\Gamma(H) \cong H$.
\end{itemize}
One could similarly define a group $\Gamma_l(H)$ by considering the left stabiliser of $H$, but it turns out that $\Gamma_l(H) \cong \Gamma(H)$. 
For more on Sch\"utzenberger groups and proofs of the above claims, see \cite[Section 2.3]{Lallement1979}.

In a commutative semigroup $S$, inclusion among principal ideals induces a partial ordering on $\mathcal{H}$ classes: $H_x \leq H_y \text{ if } xS^1 \subseteq yS^1.$
It is easy to see that for any $a,x \in S$ we have $H_{xa} \leq H_x.$ 
From this it follows that there can be at most one minimal $\mathcal{H}$-class and therefore if such an $\mathcal{H}$-class exists, it is the least $\mathcal{H}$-class under this partial ordering.
We refer to such an $\mathcal{H}$-class as \emph{the minimal $\mathcal{H}$-class}.

For Section 4, it will be important to be familiar with the basic structure theory of  commutative semigroups. 
An \emph{archimedean} semigroup is a commutative semigroup $S$ such that for each $a,b \in S$ there exists $n >0$ such that $H_{a^n} \leq H_b$. 
We say that a semigroup $S$ is a \emph{semilattice of semigroups} if for some semilattice $Y$, we can write $S$ as a disjoint union of subsemigroups $S =\bigcup_{\alpha \in Y}S_\alpha$ such that $S_{\alpha}S_{\beta} \subseteq S_{\alpha \beta}$ for all $\alpha,\beta \in Y$. 
In this case we write $S=\mathcal{S}(Y,\{S_{\alpha}\}_{\alpha \in Y})$. 
This leads us to the following structural theorem for commutative semigroups.

\begin{prop}
	\normalfont \cite[Theorem 4.2.2]{grillet1995semigroups} \textit{A commutative semigroup $S$ is a semilattice of archimedean semigroups $\mathcal{S}(Y,\{S_{\alpha}\}_{\alpha \in Y})$. 
	Furthermore, if $S$ is finitely generated then $Y$ is finite.}
\end{prop}

It follows from the definition that an archimedean semigroup can contain at most one idempotent. 
Hence a finitely generated commutative semigroup contains only finitely many idempotents. 

We call a semigroup with a zero \emph{nilpotent} if every element has a power equal to 0. 
An \emph{ideal extension} of a semigroup $S$ by a semigroup $Q$ is a semigroup $E$ such that $S$ is an ideal of $E$ and the Rees quotient $E/S$ is isomorphic to $Q$. 

The following result provides a characterisation of archimedean semigroups with idempotent. 

\begin{prop}
	\label{prop2}
	\normalfont \cite[Proposition 4.2.3]{grillet1995semigroups} \textit{A commutative semigroup $S$ is archi\-medean with idempotent if and only if $S$ is either a group or an ideal extension of a group by a nilpotent semigroup.}
\end{prop}

In general, the structure of archimedean semigroups is complex. 
For more about the decomposition of commutative semigroups into archi\-medean subsemigroups, see \cite[Chapter 4]{grillet1995semigroups}.

\section{Sch\"utzenberger Groups}
It is a natural question to ask, for each of the three generalisations of residual finiteness, whether that property is inherited by Sch\"utzenberger groups. 
This is motivated by the following result.

\begin{prop}\normalfont{\cite[Theorem 3.1]{GRAY201421}}
	\label{schutzgroupRF}
	\textit{Every Sch\"utzenberger group of a residually finite semigroup is residually finite.}
\end{prop}

 In this section we show that the properties of complete separability and strong subsemigroup separability are inherited by Sch\"utzenberger groups. 
 By way of contrast, it is not true that every Sch\"utzenberger group of a weakly subsemigroup separable semigroup is weakly subsemigroup separable. 
 However, we are able to give a sufficient condition on the stabiliser of an $\mathcal{H}$-class of a weakly subsemigroup separable semigroup to ensure that the corresponding Sch\"utzenberger group is also weakly subsemigroup separable. 
 We begin by showing that the non-group $\mathcal{H}$-classes of a strongly subsemigroup separable semigroup are finite.

\begin{prop}
	\label{lem4}
	Every non-group $\mathcal{H}$-class of a strongly subsemigroup separable semigroup is finite.
\end{prop}

\begin{proof}
	For a contradiction suppose that $S$ is a strongly subsemigroup separable semigroup with an infinite non-group $\mathcal{H}$-class $H$.
	Fix some $h \in H$ and let $T=\langle H\!\setminus\! \{h\}\rangle$. 
	If $h \in T$, then $h \in H^n$ for some $n \geq 2$, which contradicts that $H$ is not a group.
	Therefore $h \notin T$. 
	Let $\sim$ be any finite index congruence on $S$ that separates $h$ from $T$. 
	Then there exist distinct elements $x,y \in H\!\setminus\! \{h\}$ such that $x \sim y$. As $(x, h) \in \mathcal{H}$, there exists some $s \in S$ such that $xs=h$. 
	Now $h\ \sim\ ys.$  
	By Green's Lemma \cite[Lemma 2.2.4]{howie1995fundamentals}, multiplication on the right by $s$ permutes $H$, so $ys \in H\!\setminus\!\{h\} \subseteq T$. 
	Hence $\sim$ does not separate $h$ from $T$, contradicting the strong subsemigroup separability of $S$.
\end{proof}

\begin{cor}
	\label{cor2}
	Every Sch\"utzenberger group of a strongly subsemigroup separable semigroup is itself strongly subsemigroup separable.
\end{cor}

\begin{proof}
	By Proposition \ref{lem4} every Sch\"utzenberger group of a non-group $\mathcal{H}$-class is finite so is certainly strongly subsemigroup separable. 
	The Sch\"utzen\-berger group of a group $\mathcal{H}$-class $H$ is isomorphic to $H$ so is strongly subsemigroup separable by Proposition \ref{lem5}.
\end{proof}

\begin{cor}
	\label{cor3}
	Every $\mathcal{H}$-class of a completely separable semigroup is finite (and hence each of its Sch\"utzenberger group is completely separable). 
\end{cor}

\begin{proof}
	As a completely separable semigroup is also strongly subsemigroup separable, every non-group $\mathcal{H}$-class is finite by Proposition \ref{lem4}. 
	By Lemma \ref{comp.sep.group} a group is completely separable if and only if it is finite, and so it follows from Proposition \ref{lem5} that all the group $\mathcal{H}$-classes are finite.
\end{proof}

The analogue of Corollary \ref{cor2} for weak subsemigroup separability does not hold in general, as will be demonstrated in Example \ref{eg5}. 
However, it does hold for commutative semigroups. 
We deduce this from the following result.

\begin{lem}
	\label{lem12}
	Let $S$ be a weakly subsemigroup separable semigroup and let $H$ be a $\mathcal{H}$-class of $S.$ 
	If there exists an element $h \in H$ such that $ah=ha$ for all $a \in \Stab(H)$, then the Sch\"utzenberger group $\Gamma(H)$ is weakly subsemigroup separable.
\end{lem}

\begin{proof}
	If $H$ is a group then $\Gamma(H) \cong H$. 
	Hence $\Gamma(H)$ is weakly subsemigroup separable by Proposition \ref{lem5}.
	
	Now assume that $H$ is not a group. 
	For $x \in \Stab(H)$, we will denote $[x]_{\sigma_H}$ by $[x]$. 
	Let $T=\langle [x_1],[x_2],\dots,[x_n] \rangle \leq \Gamma(H)$ and let $[u] \in \Gamma(H) \!\setminus\! T$. 
	Let $h$ be as in the statement of the lemma.
	Let $\overline{T}=\langle h,x_1,x_2,\dots,x_n \rangle \leq S$.  
	
	First we will show that $hu \notin \overline{T}$. 
	For a contradiction assume that $hu \in \overline{T}$. 
	Then as $hx_j=x_jh$ for $1 \leq j \leq n$, we have $hu=h^it$ for some $t \in \langle x_1,x_2,\dots,x_n \rangle$ and $i \geq 0$. 
	We split into three cases: $i=0$, $i=1$, and $i>1$.
	
	(i) If $i=0$ we have $hu=t \in \Stab(H)$, so $h^2u \in H \cap H^2$, which contradicts that $H$ is not a group.
	
	(ii) If $i=1$ we have $hu=ht$ but $[t] \neq [u]$, contradicting $\Gamma(H)$ acting freely on $H$.
	
	(iii) Finally, assume $i > 1$. As $u, t \in \Stab(H)$, it follows that $h^it \in H \cap H^i$, contradicting that $H$ is not a group.
	
	As $S$ is weakly subsemigroup separable, there exists a finite semigroup $U$ and homomorphism $\phi:S \to U$ such that $\phi(hu) \notin \phi(\overline{T})$. 
	Let $H_{\phi(h)} \subseteq U$ be the $\mathcal{H}$-class of $\phi(h)$. 
	Now $\phi(\Stab(H))\subseteq \Stab(H_{\phi(h)})$. Consider the Sch\"utzenberger group $\Gamma(H_{\phi(h)})$. 
	Then the map $\theta: \Gamma(H) \to \Gamma(H_{\phi(h)})$ given by $\theta([v])=[\phi(v)]$ is a homomorphism. 
	If it were the case that $\theta([u]) \in \theta(T)$, then $\phi(h)\phi(u)=\phi(h)\phi(t)$ for some $t \in \langle x_1, x_2, \dots, x_n \rangle$, contradicting $\phi(hu) \notin \phi(\overline{T})$.
	Hence  $\theta([u]) \notin \theta(T)$. 
	We conclude $\Gamma(H)$ is weakly subsemigroup separable. 
\end{proof}

\begin{cor}\label{lem1}
	Every Sch\"utzenberger group of a weakly subsemigroup separable commutative semigroup is itself weakly subsemigroup separable.
\end{cor}

Although Sch\"utzenberger groups of weakly subsemigroup separable semigroups need not in general be weakly subsemigroup separable, it is an intriguing open question whether they must be weakly \emph{subgroup} separable.

\begin{op}
Let $S$ be a semigroup and let $H$ be an $\mathcal{H}$-class of $S$.  
If $S$ is weakly subsemigroup separable, is $\Gamma(H)$ weakly subgroup separable?
\end{op}

In the final part of this section we provide some partial solutions to this problem, one of which is utilized in the proof of Proposition \ref{lem10}.

\begin{prop}
\label{last}
	Let $S$ be a semigroup and let $H$ be an $\mathcal{H}$-class of $S.$  
	If $S$ is weakly subsemigroup separable, then $\Gamma(H)$ satisfies the separability property with respect to the collection of all finitely generated abelian subgroups.
\end{prop}

\begin{proof}
If $H$ is a group, then $\Gamma(H)\cong H$ is a subgroup of $S$ and hence weakly subsemigroup separable by Proposition \ref{lem5}.  
It then certainly satisfies the separability property with respect to the collection of all finitely generated abelian subgroups.

Suppose that $H$ is not a group.  
Let $G$ be a finitely generated abelian subgroup of $\Gamma(H)$ and let $b\in\Gamma(H)\!\setminus\!G.$
Now, $G$ is generated (as a group) by some set $\{a_1, \dots, a_n\}\cup G_0,$ where each $a_i$ is non-torsion and $G_0$ is the finite torsion subgroup of $G$.
For $s \in \Stab(H)$, we shall just write $[s]$ for $[s]_{\sigma_H}.$
Let $U$ denote the subsemigroup $$\{u\in\text{Stab}(H) \mid [u]\in G\}$$ of $\text{Stab}(H).$
For each $i\in\{1, \dots, n\},$ select $x_i, y_i\in U$ such that $[x_i]=a_i$ and $[y_i]=a_i^{-1}.$  
Also, fix an element $h\in H.$

The remainder of this proof is organised as follows.
\begin{enumerate}
	\item We show that for each $i\in\{1, \dots, n\},$ there exist $\alpha(i), \beta(i)\in\mathbb{N}$ such that $x_i^{\alpha(i)}h=hy_i^{\beta(i)}.$
	\item We build a finitely generated subsemigroup $T$ of $S$ such that $T\cap H=\{hu \mid u\in U\}.$
	\item We find a finite group $K$ and a homomorphism $\theta : \Gamma(H)\to K$ such that $\theta(b)\notin \theta(G).$
\end{enumerate}

(1)  Let $i\in\{1, \dots, n\},$ and write $x=x_i$, $y=y_i.$
We claim that $hy\in\langle h, x\rangle.$  
Since $S$ is weakly subsemigroup separable, it suffices to show that $hy$ cannot be separated from $\langle h, x\rangle$ by a finite index congruence.  
Suppose that $\sim$ is a finite index congruence on $S$. 
Then there exist $k, \ell\in\mathbb{N}$ with $k<\ell$ such that $hx^k \sim hx^{\ell}.$  
The elements $[x]$ and $[y]$ are inverses of each other in the group $\Gamma(H)$ and so, by the definition of $\sigma_H$, $hy=hx^ky^{k+1}$.
Hence
$$hy=hx^ky^{k+1}\sim hx^{\ell}y^{k+1}=hx^{\ell-k-1}\in\langle h, x\rangle,$$
which completes the proof of the claim. 
Since $x\in\text{Stab}(H)$ and $H$ is not a group, we cannot have $hy\in\langle x\rangle.$  
Post-multiplying $hy$ by an appropriate power of $y,$ we deduce that $hy^j=uh$ for some $j\in\mathbb{N}$ and $u\in\langle h, x\rangle.$ 

Now we claim that $hx^j\in\langle h, u\rangle$. 
Again, it suffices to show that $hx^j$ cannot be separated from $\langle h, u \rangle$ in a finite index congruence.
Suppose that $\sim$ is a finite index congruence on $S.$ 
Then there exist $k, \ell\in\mathbb{N}$ with $k<\ell$ such that $hy^{jk} \sim hy^{j\ell}.$ 
Recalling that $[x]$ and $[y]$ are mutually inverse and that $hy^j=uh$, we deduce 
$$hx^j=hy^{jk}x^{j(k+1)} \sim hy^{j\ell}x^{j(k+1)}=hy^{j(\ell-k-1)}=u^{\ell-k-1}h\in\langle h, u\rangle,$$
establishing the claim.
Since $uh\in H,$ we cannot have $hx^j\in\langle u\rangle,$ for then $(hx^j)h\in H\cap H^2,$ contradicting that $H$ is not a group.
It follows that $hx^j=w_1hw_2$ for some $w_1\in\langle u\rangle^1$ and $w_2\in\langle h, u\rangle^1\subseteq\langle h, x\rangle^1.$  
Now $w_1h=hy^{jk}$ for some $k\in\mathbb{N}_0,$ so $hx^j\neq w_1h$ as $[x]$ has infinite order in $\Gamma(H),$ and hence $w_2\neq 1.$ 
We claim we cannot write $w_2$ as $sht$ for some $s\in\langle x\rangle^1$ and $t\in\langle h, x\rangle^1.$  
Indeed, if we could, then since $w_1h\in H$ and $x\in\text{Stab}(H),$ we would have $hx^j\in H\cap H^n$ for some $n\geq 2,$ contradicting that $H$ is not a group. 
Hence the claim holds. 
We must then have $w_2\in\langle u\rangle.$  
But $u\in \langle h, x\rangle,$ so we conclude that $u=x^m$ for some $m\in\mathbb{N}.$  
Now set $\alpha(i)=m$ and $\beta(i)=j.$
We fix $\alpha(i)$ and $\beta(i)$ for the remainder of this proof.

(2)  For each $i\in\{1, \dots, n\}$ let $m_i=\max(\alpha(i), \beta(i)).$
For each $g\in G_0,$ select $u_g\in U$ such that $[u_g]=g.$
We define a finite set 
$$W=\{x_1^{j_1}\dots x_n^{j_n} \mid 0\leq j_i\leq m_i-1\text{ for }1\leq i\leq n\}\cup\{u_g \mid g\in G_0\} \subseteq U.$$
Let $X=\{x_i^{\alpha(i)} \mid 1\leq i\leq n\},$ and let $T$ be the subsemigroup of $S$ generated by
$$Z=X\cup\{hw \mid w\in W\}.$$
Note that $U \cap Z =X$ and $H \cap Z=\{hw \mid w\in W\}$.
We prove that $T\cap H=\{hu \mid u\in U\}.$

First, let $h^{\prime}\in T\cap H.$  Then $h^{\prime}=z_1\dots z_k$ for some $z_j\in Z.$  
If every $z_j\in X,$ then $h^{\prime}\in\text{Stab}(H),$ contradicting that $H$ is not a group.  
Therefore, there exists $j$ minimal such that $z_j=hw$ for some $w\in W.$   
Then for each $i<j,$ we have $z_i\in X$ and hence $z_ih\in hU$, as $xh \in hU$ for all $x \in X$ by (1).
So we deduce that
$h^{\prime}=hw^{\prime}z_{j+1}\dots z_k$ for some $w^{\prime}\in U.$
Let $u=w^{\prime}z_{j+1}\dots z_k.$
We shall show that $u \in U$.
Suppose that $H \cap \{z_{j+1}, z_{j+2}, \dots, z_k\}=\{z_{i_1}, \dots, z_{i_m}\}$ where $j+1\leq i_1<\dots<i_m\leq k.$
Let $h_{\ell}=z_{i_{\ell}}\dots z_{i_{\ell+1}-1}$ for $\ell\in\{1, \dots, m-1\},$ and let $h_m=z_{i_m}\dots z_k.$  Then $h_{\ell}\in H$ for each $\ell\in\{1, \dots, m\}.$  
But then
$$h^{\prime}=(hw^{\prime}z_{j+1}\dots z_{i_1-1})h_1\dots h_m\in H\cap H^{m+1},$$ which contradicts that $H$ is not a group.  
We conclude that $z_i\in X$ for every $i\in\{j+1, \dots, k\}.$ 
It follows that $u\in U$ and hence $h^{\prime}\in\{hu \mid u\in U\}.$

For the reverse containment, let $u\in U.$  Since $G$ is abelian, we have $[u]=a_1^{k_1}\dots a_n^{k_n}c$ for some $k_i\in\mathbb{Z}$ and $c\in G_0.$
Consider $i\in\{1, \dots, n\}.$  
If $k_i\geq 0,$ let $p_i\in\mathbb{N}_0$ and $r_i\in\{0, \dots, \alpha(i)-1\}$ be such that $k_i=p_i\alpha(i)+r_i,$ and let $q_i=s_i=0.$
If $k_i<0,$ let $q_i\in\mathbb{N}$ and $s_i\in\{0, \dots, \beta(i)-1\}$ be such that $k_i=-q_i\beta(i)+s_i,$ and let $p_i=r_i=0.$  Now let $t_i=\max(r_i, s_i).$
It follows that $$hu=\bigl(x_1^{\alpha(1)}\bigr)^{q_1}\dots\bigl(x_n^{\alpha(n)}\bigr)^{q_n}\bigl(hx_1^{t_1}\dots x_n^{t_n}u_c\bigr)\bigl(x_1^{\alpha(1)}\bigr)^{p_1}\dots\bigl(x_n^{\alpha(n)}\bigr)^{p_n},$$ 
so $hu\in T\cap H,$ and hence $T\cap H=\{hu \mid u\in U\}.$

(3)  Choose $v\in\text{Stab}(H)$ such that $[v]=b.$  
Then $hv\notin T$ by (2). 
Since $S$ is weakly subsemigroup separable, there exists a finite semigroup $P$ and a homomorphism $\phi : S\to P$ such that $\phi(hv)\notin\phi(T).$  
Let $H_{\phi(h)}$ denote the $\mathcal{H}$-class of $\phi(h),$ and let $K$ be the finite group $\Gamma(H_{\phi(H)})$.  
As in the proof of Lemma \ref{lem12}, the map $\theta: \Gamma(H)\to K,$ given by $\theta([t])=[\phi(t)],$ is a homomorphism such that $\theta(b)\notin \theta(G),$ as required.
\end{proof}

\begin{cor}
\label{abelianSgroup}
	Let $S$ be a semigroup and let $H$ be an $\mathcal{H}$-class of $S.$  
	If $S$ is weakly subsemigroup separable and $\Gamma(H)$ is abelian, then $\Gamma(H)$ is weakly subgroup separable.
\end{cor}

Groups satisfying the separability property with respect to the collection of all cyclic subgroups are known in the literature as {\em cyclic subgroup separable groups} or {\em $\Pi_c$ groups}.  
Such groups have received considerable attention; see for instance \cite{burillo, kim, Sokolov, Wong}.

From Proposition \ref{last} we immediately deduce:

\begin{cor}
	Let $S$ be a semigroup and let $H$ be an $\mathcal{H}$-class of $S$.  If $S$ is weakly subsemigroup separable, then $\Gamma(H)$ is cyclic subgroup separable.
\end{cor}

In Section 6 we will return to Sch\"utzenberger groups and consider the following question: if all the Sch\"utzenberger groups of a semigroup have a separability property, does the semigroup itself have that property? 

\section{Finitely Generated Commutative Semigroups}
In this section we give a characterisation of our separability properties in finitely generated commutative semigroups. 
We first note that finitely generated abelian groups are strongly \emph{subgroup} separable.
In fact, Mal'cev proved, more generally, that every polycyclic-by-finite group is strongly subgroup separable \cite{Malcev}.  
However, when it comes to finitely generated commutative semigroups, Example \ref{eg1} shows that there exist finitely generated commutative semigroups which are not even weakly subsemigroup separable.
The question of when a finitely generated commutative semigroup is completely separable or strongly subsemigroup separable was considered by Kublanovski{\v \i} and Lesohin in \cite{MR541122}. 
We briefly outline their setup and results, without giving proofs.

Let $S$ be a finitely generated commutative semigroup with finite generating set $A$. 
For $s \in S$, let $C_s=A \cap \Stab(H_s)$. 
Then $C_s$ is finite and can be empty. 
We denote $|C_s|$ by $k_s$. 
Then $\langle C_s \rangle^1=\Stab(H_s)$. 
Consider the free commutative monoid  $\mathbb{N}_0^{k_s}$ on $k_s$ generators. 
There is a canonical homomorphism $\phi: \mathbb{N}_0^{k_s} \to \Stab(H_s).$ We note that $\Stab(s)$, the point stabiliser of $s$, is a submonoid of $\Stab(H_s)$. 
Let $W_s=\phi^{-1}(\Stab(s)) \leq \mathbb{N}_0^{k_s}$ be the pre-image of $\Stab(s)$. 
We can view $\mathbb{N}_0^{k_s}$ as a submonoid of the free abelian group $\mathbb{Z}^{k_s}$. 
Consider the subgroup $G_s\leq \mathbb{Z}^{k_s} $ generated by $W_s$. 
As $G_s$ is a subgroup of $\mathbb{Z}^{k_s}$, we have $G_s \cong \mathbb{Z}^{m_s}$ for some $m_s \leq  k_s$.  
Using this, Kublanovski{\u \i } and Lesohin were able to characterise when $S$ is strongly subsemigroup separable as follows.

\begin{thm} \normalfont \cite[Theorem 1]{MR541122}
	\label{m=k}
	\textit{A finitely generated commutative semigroup $S$ is strongly subsemigroup separable if and only if $m_s=k_s$ for all $s \in S$.}
\end{thm}

From the proof of this result they obtained two corollaries, summarised as follows: 

\begin{cor} \normalfont \cite[Corollaries 2 and 3]{MR541122}
	\label{cor5}
	\textit{For a finitely generated commutative semigroup $S$ the following are equivalent:
	\begin{enumerate}
		\item[\textup{(1)}] $S$ is completely separable;
		\item[\textup{(2)}] $S$ is strongly subsemigroup separable;
		\item[\textup{(3)}] if $a,b \in S$ such that $a \in b^nS$ for all $n \in \mathbb{N}$, then there exists $m \in \mathbb{N}$ such that $a=b^ma$.
	\end{enumerate}}
\end{cor}

We enhance the result of Corollary \ref{cor5} by showing that for a finitely generated commutative semigroup, weak subsemigroup separability is also equivalent to complete separability. 
We also provide a new characterisation in terms of $\mathcal{H}$-classes.
This characterisation removes the need for the parameters $k_s$ and $m_s$ of Theorem \ref{m=k}. 
The proof we provide is independent of the work of Kublanovski{\u \i } and Lesohin, although the reader may note parallels between the methods used.

We now state the main result of this paper.

\begin{thm}
	\label{thm1}
	Let $S$ be a finitely generated commutative semigroup. Then the following are equivalent:
	\begin{enumerate}
		\item[\textup{(1)}] $S$ is completely separable;
		\item[\textup{(2)}] $S$ is strongly subsemigroup separable;
		\item[\textup{(3)}] $S$ is weakly subsemigroup separable;
		\item[\textup{(4)}] every $\mathcal{H}$-class of $S$ is finite.
	\end{enumerate}
\end{thm}

In order to prove Theorem \ref{thm1}, we first establish some notation and prove some preliminary results. 

For an $\mathcal{H}$-class $H$ of a finitely generated commutative semigroup $S$ define 
\[I(H)=\bigcup\{H_s \mid s \in S, H_s \ngeq H\}.\]
Note that $I(H)$ is non-empty if and only if $H$ is not the minimal $\mathcal{H}$-class, in which case $I(H)$ is an ideal. 
In the proofs of the results of this section, when dealing with a non-minimal $\mathcal{H}$-class $H$ we shall often pass to the Rees quotient $S/I(H)$. 
The following observations are needed to justify this strategy.

Let $H$ be a non-minimal $\mathcal{H}$-class, and denote $I(H)$ by $I$.  
Recall that $S/I$ consists of a zero element, namely $I$, and singleton sets $\{s\}$ where $s\in S\!\setminus\!I$.  
For an $\mathcal{H}$-class $H^{\prime}$ of $S$, let $H^{\prime}_I=\{[h]_I \mid h \in H^{\prime}\}$ denote the image of $H^{\prime}$ in $S/I$. 
The following statements can easily be verified.
\begin{itemize}
	\item For an $\mathcal{H}$-class $H^{\prime}$ of $S$, the set $H^{\prime}_I$ is an $\mathcal{H}$-class of $S/I$.
	\item For an $\mathcal{H}$-class $H^{\prime}$ of $S$, we have that $H^{\prime} \geq H$ if and only if $H^{\prime}_I \geq H_I$.
	\item The set $H_I$ is the unique minimal non-zero $\mathcal{H}$-class in $S/I$.
	\item For $x \in S$, we have that $x \in \Stab(H)$ if and only $[x]_{I} \in \Stab(H_{I})$.
\end{itemize}

\begin{lem}
	\label{lem2}
	Let $S$ be a finitely generated commutative semigroup and let $H$ be an $\mathcal{H}$-class. 
	Fix $s \in S$ with $H_s \geq H$. 
	Then \[\Stab(H)=\{x \in S^1 \mid H_{sx^n} \geq H \text{ for all } n \in \mathbb{N}\}.\]
\end{lem}

\begin{proof}
	First we assume that $H$ is the minimal $\mathcal{H}$-class in $S$. 
	In this case $S^1=\Stab(H)$ and $H_{sx^n} \geq H$ for all $s \in S$, $x \in S^1$ and $n \in \mathbb{N}$. 
	Hence the result follows. 
	
	Now assume that $H$ is a non-minimal $\mathcal{H}$-class. 
	By noticing that $I(H)\cap \Stab(H)=\emptyset$  and 
	\[I(H) \cap \{x \in S^1 \mid H_{sx^n} \geq H \text{ for all } n \in \mathbb{N}\} =\emptyset,\] and factoring out by $I(H)$, we may assume that $H$ is the unique minimal non-zero $\mathcal{H}$-class.
	
	Let $x \in \Stab(H)$. 
	For a contradiction assume that there exists $n \in \mathbb{N}$ such that $H_{sx^n}\ngeq H$. 
	Then, as $H$ is the minimal non-zero $\mathcal{H}$-class, we have $sx^n=0$. 
	As $H_s \geq H$, there exists $t \in S^1$ such that $st \in H$. 
	Then \[0=sx^nt=stx^n \in Hx^n=H,\] which is a contradiction. 
	
	Now assume that $H_{sx^n} \geq H$ for all $n \in \mathbb{N}$. 
	Fix $h \in H$. Assume for a contradiction that $x \notin \Stab(H)$. 
	Then $hx=0$. 
	As $S$ is finitely generated and commutative, it is residually finite (see \cite[Theorem 3]{Lallement1971}). 
	Let $\sim$ be an arbitrary finite index congruence on $S$. Then there exist $m,n \in \mathbb{N}$, with $m < n$, such that $sx^m \sim sx^n$. 
	As $H_{sx^m} \geq H$, there exists $t \in S^1$ such that $sx^mt=h$. 
	Then \[h=sx^mt \sim sx^nt=sx^mtx^{n-m}=hx^{n-m}=0.\] 
	As $\sim$ is arbitrary, we have shown we cannot separate $h$ and 0 in a finite quotient. 
	This contradicts $S$ being residually finite.
\end{proof}

\begin{cor}
\label{arch}
	Let $S$ be a finitely generated commutative semigroup and let $H$ be a non-minimal $\mathcal{H}$-class. 
	Let $I=I(H)$. 
	If $A$ is an archimedean component in $S/I$ not containing the zero element, then $A \subseteq \Stab(H_{I})$.
\end{cor}

\begin{proof}
Let $a \in A$. 
As $A$ is a subsemigroup we have $a^n \in A$ for all $n \in \mathbb{N}$. 
So $H_{a \cdot a^n} \geq H_{I}$ for all $n \in \mathbb{N}$. 
Hence $a \in \Stab(H_{I})$ by Lemma \ref{lem2}.
\end{proof}

Now we are ready to prove Theorem \ref{thm1}.

\begin{proof}[Proof of Theorem \ref{thm1}]
	
	From Proposition \ref{lem} we know that (1) implies (2) and (2) implies (3). 
	We only need to show that (3) implies (4) and that (4) implies (1).
	
	(3) $\Rightarrow$ (4). 
	For a contradiction assume that $S$ is weakly subsemigroup separable but has an infinite $\mathcal{H}$-class $H$. 
	Then the Sch\"utzenberger group $\Gamma(H)$ must also be infinite. 
	As $H_{xa} \leq H_x$ for all $x,a \in S$, the complement of $\Stab(H)$ is an ideal of $S$. 
 	It follows that if $X$ is a finite generating set for $S$, then $\Stab(H)$ is generated by $\Stab(H) \cap X$ and so is also finitely generated.
	As $\Gamma(H)$ is a quotient of a finitely generated commutative monoid, it is a finitely generated abelian group. 
	As $\Gamma(H)$ is infinite, it must contain a subgroup isomorphic to $\mathbb{Z}$. 
	So $\Gamma(H)$ is not weakly subsemigroup separable by Example \ref{eg1} and Proposition \ref{lem5}. 
	However, this contradicts Corollary \ref{lem1}, so there cannot be an infinite $\mathcal{H}$-class.
	
	(4) $\Rightarrow$ (1). 
	Let $h \in S$. 
	We shall show that there exists a finite index congruence on $S$ such that the congruence class of $h$ is a singleton. 
	For $s \in S$, let $H_s$ denote the $\mathcal{H}$-class of $s$ and $A_s$ the archimedean component of $s$. 
	For convenience, we will use $H$ to denote $H_h$ and $\sigma$ to denote the Sch\"utzenberger congruence $\sigma_H$.
	Let $I=I(H).$
	
	We split into two cases.
	
	\textbf{Case 1}. 
	The first case is that $H$ is not a group. 
	In particular, $H$ is not the minimal $\mathcal{H}$-class. 
	Factoring out $I$, we may assume that $H$ is the unique minimal non-zero $\mathcal{H}$-class in $S$.
	As $H$ is not a group, then $H^2=\{0\}$. 
	So $A_h=A_0$ and $A_h$ is a nilpotent semigroup. 
	It follows from Corollary \ref{arch} that $S^1$ is the disjoint union of $\Stab(H)$ and $A_h$.
	Consider any finite generating set for $S$, and write it as $X\cup Y$, where $X\subseteq\Stab(H)$ and $Y\subseteq A_h$. 
	Then $\langle X \rangle \subseteq \Stab(H)$ and, as $A_h$ is non-empty, it must be the case that $Y$ is non-empty. 
	Note that $U= \langle Y \rangle$ is finite as $A_h$ is nilpotent. 
	We may assume that $X$ is non-empty, for otherwise $S=U$ is a finite semigroup and hence certainly completely separable.
	
Let $X=\{x_1, x_2, \dots, x_m\}$ and let $\overline{X}=\{\overline{x}_1, \overline{x}_2, \dots, \overline{x}_m\}$ be disjoint from $X$.
Let $\FC_{\overline{X}}$ denote the free commutative monoid on $\overline{X}$. 
Let $\phi: \FC_{\overline{X}} \to \Stab(H)$ be the unique extension to a homomorphism of the map given by $\overline{x_i} \to x_i$.
For $u \in U$ define
\[I_u=\{w \in \FC_{\overline{X}} \mid u\phi(w)\in H\}.\] 
Suppose that $w\in I_u$ and $z\in \FC_{\overline{X}}$. 
Then since $u\phi(w)\in H$ and $\phi(z)\in\text{Stab}(H)$, we deduce that $u\phi(wz)=\left(u\phi(w)\right)\phi(z)\in H$. 
Thus, if $I_u$ is non-empty, then it is an ideal if $\FC_{\overline{X}}$.
	
The monoid $\FC_{\overline{X}}$ is isomorphic to $\mathbb{N}_0^{|X|}$. 
Ideals of $\mathbb{N}_0^{|X|}$ are upward closed sets under the component-wise ordering on tuples. 
It is well known that this partially ordered set has no infinite antichains (Dickson's Lemma). 
From this follows the well-known fact that every ideal of $\FC_{\overline{X}}$ is finitely generated as an ideal.
	
Let $U^{\prime}=\{u\in U \mid I_u\neq\emptyset\}$.  
For each $u\in U^{\prime}$ let $Z_u$ be a finite generating set for $I_u$, and let 
$$Z=\bigcup_{u \in U^{\prime}}Z_u.$$ 
As $U$ is finite, we have that $Z$ is finite. 
For each $z\in Z$, we have 
$$z={\overline{x}_1}^{\alpha_1(z)}{\overline{x}_2}^{\alpha_2(z)}\dots {\overline{x}_m}^{\alpha_m(z)}$$ for some $\alpha_i(z)\in\mathbb{N}_0$. 
Define
\[n=\text{max}\{\alpha_i(z) \mid z\in Z, 1\leq i\leq m\}.\]
	
Now let $\sim$ be the congruence on $S$ generated by the set \linebreak $\{(x_i^n,x_i^{n+|H|}) \mid 1 \leq i \leq m \}.$
Each of the finitely many generators of $S/\!\!\sim$ is periodic and $S/\!\!\sim$ is commutative, so $S/\!\!\sim$ is finite.
We now show that $[h]_{\sim}=\{h\}$. 
	
Let $t \sim h$.  
We need to show that $t=h$.  
Clearly it is sufficient to assume that $t$ is obtained from $h$ by a single application of a pair from the generating set of $\sim$. 
So, let $h=sx_i^p$ and $t=sx_i^q$ where $1 \leq i \leq m$, $s \in S$ and $\{p,q\} = \{n,n+|H|\}$.
	
If $(p,q)=(n,n+|H|)$ then $t=hx_i^{|H|}$. 
Since $x_i \in \Stab(H)$, $[x_i]_{\sigma}$ is an element of the Sch\"utzenberger group $\Gamma(H)$. 
As $|\Gamma(H)|=|H|$, it follows that $[x_i^{|H|}]_{\sigma}=[x_i]^{|H|}_{\sigma}=[1]_{\sigma}$. 
Hence $t=hx_i^{|H|}=h$.
	
Now we consider the case when $(p,\ q)=(n+|H|, n)$. 
As $h \notin \Stab(H)$, we have $s\in A_h\!\setminus\!\{0\}$. 
Any way of decomposing $s$ into generators must contain at least one element from $Y$. 
Therefore, we have that $s=us'$, where $u \in U'$ and $s' \in \Stab(H)$. 
Fix some \[w={\overline{x}_1}^{\beta_1}{\overline{x}_2}^{\beta_2}\dots {\overline{x}_m}^{\beta_m} \in \FC_{\overline{X}}\] such that $\phi(w)=s'$. 
As $h=us'x_i^{n+|H|} \in H$, we have that 
\[w{\overline{x}_i}^{n+|H|}={\overline{x}_1}^{\beta_1}{\overline{x}_2}^{\beta_2}\dots {\overline{x}_i}^{\beta_i+n+|H|} \dots {\overline{x}_m}^{\beta_m} \in I_u.\]  
Then there exist $z \in Z$ and $w'={\overline{x}_1}^{\gamma_1} \dots {\overline{x}_m}^{\gamma_m} \in \FC_{\overline{x}}$ such that 
\[zw'=w{\overline{x}_i}^{n +|H|}.\] 
For $1 \leq j \leq m$, we have that 
\[ \alpha_j(z)+\gamma_j=\begin{cases} \beta_{j} &\text{if } j \neq i, \\ \beta_i+n+|H| &\text{if } j=i. \end{cases} \]
As $\alpha_i(z) \leq n$ by definition, it must be the case that $\gamma_i \geq |H|$. 
Then
\[w{\overline{x}_i}^n=z{\overline{x}_1}^{\gamma_1}\dots {\overline{x}_i}^{\delta} \dots {\overline{x}_m}^{\gamma_m}\] where $\delta=\gamma_i-|H| \geq 0$. 
Hence $w{\overline{x}_i}^n \in I_u$ and so $sx_i^n=t \in H$. 
As $h=tx_i^{|H|}$, a similar argument as above proves that $h=t$, as required.
	
\textbf{Case 2}. Now we assume that $H$ is a group. 
By Proposition \ref{prop2}, $A_h$ is either a group or the ideal extension of $H$ by a nilpotent semigroup. 
Hence $A_h \subseteq \Stab(H)$. 
If $H$ is the minimal $\mathcal{H}$-class of $S$, then $S^1=\text{Stab}(H).$  
If $H$ is not minimal, we may assume that it is the unique minimal non-zero $\mathcal{H}$-class of $S$ (by taking the Rees quotient by $I$), in which case we have $S^1\!\setminus\!\{0\}=\text{Stab}(H)$ by Corollary \ref{arch} and the fact that $A_h \subseteq \Stab(H)$.  
	
In either case, let $X$ be a finite generating set for $\Stab(H)$. 
As in Case 1, let $\overline{X}$ be a set in bijection with $X$, let $\FC_{\overline{X}}$ denote the free commutative monoid on $\overline{X}$, and let $\phi: \FC_{\overline{X}} \to \Stab(H)$ be the unique extension to a homomorphism of a bijection $\overline{X} \to X$.
Define 
\[J=\{w \in \FC_{\overline{X}} \mid \phi(w) \in H\}.\] 
Then $J$ is an ideal of $\FC_{\overline{X}}$. 
Let $Z$ be a finite generating set for $J$ (as an ideal), and let $n$ be the maximal exponent of a generator appearing in any word $z \in Z$.
	
Let $\sim$ be the congruence on $S$ with generating set $\{(x^n,x^{n+|H|}) \mid x \in X\}.$ 
An argument essentially the same as that of Case 1 shows that $S/\!\!\sim$ is finite and $[h]_\sim=\{h\}$, completing the proof of this direction and of the theorem.
\end{proof}

\section{Beyond Finitely Generated Commutative Semigroups}

Given that for finitely generated commutative semigroups, the three properties of complete separability, strong subsemigroup separability and weak subsemigroup separability coincide, the following questions naturally arise.
\begin{itemize}
	\item For commutative semigroups in general (not necessarily finitely generated), do the properties of complete separability and strong subsemigroup separability coincide?
	\item For commutative semigroups in general, do the properties of strong subsemigroup separability and weak subsemigroup separability coincide?
	\item For finitely generated semigroups in general (not necessarily commutative), do the properties of complete separability and strong subsemigroup separability coincide?
	\item For finitely generated semigroups in general, do the properties of strong subsemigroup separability and weak subsemigroup separability coincide?
\end{itemize}

In this section we answer all of these questions in the negative. 
We will first deal with the commutative case and then the finitely generated case.

\subsection{Non-finitely Generated Commutative Semigroups}

We give an example of a commutative semigroup that is weakly subsemigroup separable but not strongly subsemigroup separable. 
In order to do this, we first establish the following result.

\begin{prop}
	\label{lem3}
	If a residually finite semigroup $S$ has $\mathbb{N}$ as a homomorphic image then it is weakly subsemigroup separable.
\end{prop}

\begin{proof}
	Let $T \leq S$ be finitely generated and let $x \in S \! \setminus \! T$. 
	By assumption there exists a homomorphism $f: S \to \mathbb{N}$. Let $n=f(x)$. 
	The set $I=\{m \mid m> n\} \subseteq \mathbb{N}$ is an ideal of $\mathbb{N}$. 
	Let $g: \mathbb{N} \to \mathbb{N}/I$ be the canonical homomorphism.

	Since $T$ is finitely generated and $f(st) >f(s)$ for any $s,t \in S$, it follows that the set 
	\[Y=\{t \in T \mid f(t)=n\}\]
	is finite. 
	Since $S$ is residually finite, there exists a finite semigroup $P$ and homomorphism $h: S \to P$ such that $h(x) \notin h(Y)$. 
	Then $(g \circ f) \times h: S \to \mathbb{N}/I \times P$ separates $x$ from $T$.
\end{proof}	

\begin{eg}
	\normalfont Consider $S=\mathbb{N} \times \mathbb{Z}$. 
	Now, $S$ is residually finite since it is the direct product of two residually finite semigroups.
	As the projection map onto the first factor gives a homomorphic image which is $\mathbb{N}$, we conclude that $S$ is weakly subsemigroup separable by Proposition \ref{lem3}. 
	
	We now show that $S$ is not strongly subsemigroup separable. 
	Consider $\mathbb{N} \times \mathbb{N} \leq S$ and the element $(2,0) \notin \mathbb{N} \times \mathbb{N}$. 
	Let $\sim$ be a finite index congruence on $S$. 
	Then there exist $i, j \in \mathbb{Z}$ with $i<j$ such that $(1, i) \sim (1, j)$. 
	Then 
	\[(2, 0)=(1, i)(1, -i) \sim (1,j)(1, -i)=(2, j-i) \in \mathbb{N} \times \mathbb{N}.\] 
	Hence, $S$ is a commutative semigroup which is weakly subsemigroup separable but not strongly subsemigroup separable. 
\end{eg}

\begin{rem}
	\normalfont The semigroup $\mathbb{N} \times \mathbb{Z}$ is also an example of a weakly subsemigroup separable semigroup which is a direct product of two semigroups where one of the factors is not weakly subsemigroup separable.
	 This is in contrast to the situation for  residual finiteness: the direct product of two semigroups is residually finite if and only if both factors are residually finite \cite[Theorem 2]{Gray2009}.
\end{rem}

We are left to find an example of a strongly subsemigroup separable commutative semigroup which is not completely separable. 
Our example is a group.

\begin{eg}
	\normalfont Let $C_2$ denote the cyclic group of order 2. 
	Let $G=C_2^{\mathbb{N}}$ be the Cartesian product of countably many copies of $C_2$. 
	By Lemma \ref{comp.sep.group}, $G$ is not completely separable. 
	But from Theorem \ref{thm3}, an abelian group is strongly subsemigroup separable if and only if it is torsion and for each prime $p$, the primary $p$-component is bounded in the exponent. 
	As every non-identity element in $G$ has order 2, $G$ certainly satisfies these conditions. 
	Hence $G$ is strongly subsemigroup separable. 
\end{eg}

\subsection{Finitely Generated Semigroups} 

In the previous section we showed that in the statement of Theorem \ref{thm1}, being commutative is not on its own a sufficient condition. 
In this section we will show that being finitely generated is also not on its own a sufficient condition. 
That is, we provide two examples of finitely generated semigroups, one of which is weakly subsemigroup separable but not strongly subsemigroup separable, and the other strongly subsemigroup separable but not completely separable.

First we give an example of a finitely generated semigroup which is weakly subsemigroup separable but not strongly subsemigroup separable. 
We do this by introducing a construction of semigroups and establishing necessary and sufficient conditions for this construction to be weakly subsemigroup separable and finitely generated. 
For the construction and proof, we will use the following notation. 
For a subset $Z \subseteq G$ of an abelian group $G$, let $X_Z=\{x_z \mid z \in Z\}$ be a set disjoint from $G$.

\begin{con}
	\label{con1}
	\normalfont Let $T$ be a semigroup, and let $G$ be an abelian group such that there exists a surjective homomorphism $\phi: T \to G$. Let $N=X_G \cup \{0\}$ be a null semigroup disjoint from $T$.
	Let $\mathcal{S}(T,G,\phi)=T \cup N$, with multiplication inherited from $T$ and $N$, and for $t \in T$ and $x_g \in X_G$ we define the following multiplication:
	\begin{equation*}
	\begin{split}
	& x_g\cdot t=x_{g \phi(t)}, \\
	& t \cdot x_g=x_{g (\phi(t))^{-1}}, \\
	& t \cdot 0=0 \cdot t =0. \\
	\end{split}
	\end{equation*}
	An exhaustive check confirms this multiplication is associative and therefore $\mathcal{S}(T,G,\phi)$ is a semigroup.
\end{con}

\begin{rem}
\label{rem3}
In Construction \ref{con1}, the set $X_G$ forms a non-group $\mathcal{H}$-class and the Sch\"utzenberger group of this $\mathcal{H}$-class is isomorphic to $G$.
\end{rem}

To give necessary and sufficient conditions for $\mathcal{S}(T,G, \phi)$ to be weakly subsemigroup separable we use the following lemma.

\begin{lem}
	\label{lem11}
	Let $G$ be a weakly subgroup separable group. 
	Let $H$ be a finitely generated subgroup of $G$, let \[U=\bigcup_{i=1}^n{Hg_i}\] be a finite union of cosets of $H$, and let $x \in G \! \setminus \! U$. 
	Then $x$ can be separated from $U$.
\end{lem}

\begin{proof}
	Let $i \in\{1,\dots,n\}$. 
	First we will show that $x$ can be separated from $Hg_i$. 
	As $x \notin Hg_i$, we have $xg_i^{-1} \notin H$. 
	As $G$ is weakly subgroup separable, there exists a finite group $K_i$ and homomorphism $\phi_i: G \to K_i$ such that $\phi_i(xg_i^{-1})  \notin \phi_i(H)$. 
	It follows that $\phi_i(x) \notin \phi_i(Hg_i)$. 
	Then \[\phi_1 \times \phi_2 \times \dots \times \phi_n : G \to K_1 \times K_2 \times \dots \times K_n\] is a homomorphism into a finite group that separates $x$ from $U$.
\end{proof}

\begin{prop}
	\label{lem10}
	Let $T$ be a semigroup, let $G$ be an abelian group such that there exists a surjective homomorphism $\phi: T \to G$, and let $S=\mathcal{S}(T,G,\phi)$. 
	Then $S$ is weakly subsemigroup separable if and only if  $T$ is weakly subsemigroup separable and $G$ is weakly subgroup separable.
\end{prop}

\begin{proof}
$(\Rightarrow)$ First assume that $S$ is weakly subsemigroup separable. 
Since $T$ is a subsemigroup of $S$, it must be weakly subsemigroup separable by Proposition \ref{lem5}.
Since $G$ is abelian and isomorphic to a Sch\"utzenberger group of $S$, it follows from Corollary \ref{abelianSgroup} that $G$ is weakly subgroup separable.
	
\vspace{\baselineskip}	

$(\Leftarrow)$ Now assume that $T$ is weakly subsemigroup separable and $G$ is weakly subgroup separable. 
Let $Y \subseteq S$ be a finite set, $U=\langle Y \rangle \leq S$ and $v \in S \! \setminus \! U$. 
Let $N\subseteq S$ be as in Construction \ref{con1}. 
Let $Y_1=Y \cap T$ and $Y_2=Y \cap N$. 
We split into cases.

	\noindent \textbf{Case 1.} Assume that $v \in T$. 
	Note that $T \cap U=\langle Y_1 \rangle$. 
	As $T$ is weakly subsemigroup separable and $v \notin \langle Y_1 \rangle$, there exists a finite semigroup $P$ and homomorphism $f: T \to P$ such that $f(v) \notin f(T \! \setminus \! \langle Y_1 \rangle)$. 
	Define $\overline{f}: S \to P^0$ by 
	\begin{equation*}
	s \mapsto \begin{cases}
	f(s) & \text{if } s \in T, \\
	0 &\text{otherwise.}
	\end{cases}
	\end{equation*}
	Then $\overline{f}$ is a homomorphism and $\overline{f}(v) \notin \overline{f} (U)$.
	
	\vspace{\baselineskip}
	
	\noindent \textbf{Case 2.} Now assume that $v\in N$ and $Y_2 = \emptyset$. 
	Then $U \subseteq T$. 
	Let $\sim$ be the congruence on $S$ with classes $T$ and $N$. 
	Then $[v]_{\sim} \neq [u]_{\sim}$ for all $u \in U$.
	
	\vspace{\baselineskip}
	
	\noindent \textbf{Case 3.} Finally assume that $v \in N $ and $Y_2 \neq \emptyset$. 
	Note that $0\in Y_2^2\subseteq N$, and hence $v\neq 0$. 
	Let $v=x_g$ and $Y_2=\{x_{g_1},\dots,x_{g_n}\}$. 
    Let $H \leq G$ be the subgroup generated by the set $\phi(Y_1)$. 
	Then \[U \cap N =X_Z \cup \{0\},\] where $Z=\bigcup_{i=1}^n{Hg_i}.$ 
	As $v \notin U$, it follows that $g \notin \bigcup_{i=1}^n{Hg_i}$. 
	As $G$ is weakly subgroup separable there exists a finite group $K$ and homomorphism $f: G \to K$ such that $f(g) \notin \bigcup_{i=1}^n{f(Hg_i)}$ by Lemma \ref{lem11}.  
	Let $P=\mathcal{S}(K,K,\text{id})=K\cup X_K\cup \{0\}$. Let $\overline{f}: S \to P$ be given by
	\begin{equation*}
	s \mapsto \begin{cases}
	(f \circ \phi)(s) & \text{if } s \in T, \\
	x_{f(g)} & \text{if } s=x_g \text{ for some } g \in G, \\
	0 & \text{if } s=0.
	\end{cases}
	\end{equation*}
	Then it is straightforward to check that $\overline{f}$ is a homomorphism with $\overline{f}(v) \notin \overline{f}(U)$. 
\end{proof}

The next lemma provides necessary and sufficient conditions for  $\mathcal{S}(T,G,\phi)$ to be finitely generated.

\begin{lem}
\label{lem13}
Let $T$ be a semigroup, let $G$ be an abelian group such that there exists a surjective homomorphism $\phi: T \to G$, and let $S=\mathcal{S}(T,G,\phi)$. 
Then $S$ is finitely generated if and only if $T$ is finitely generated.
\end{lem}

\begin{proof}
If $S$ is finitely generated, then as $T$ is the complement of an ideal, it must also be finitely generated. 
Conversely, if $T$ is generated by a finite set $Y$, then it is easy to see that $S$ is generated by $Y\cup\{x_{1_G}\}$.
\end{proof}

We provide an example of a weakly subsemigroup separable semigroup $S$ that has the following properties:
\begin{itemize}
	\item $S$ is finitely generated, non-commutative, but not strongly subsemigroup separable; 
	\item $S$ has a Sch\"utzenberger group which is not weakly subsemigroup separable.	
\end{itemize}

\begin{eg}
	\label{eg5}
	Let $F_2=\{a,b\}^+$ be the free semigroup on $\{a, b\}$. 
	Let $\phi: F_2 \to \mathbb{Z}$ be given by $a \mapsto 1$ and $b \mapsto -1$. 
	As $F_2$ is completely separable \cite[Corollary 1]{MR0274613} and $\mathbb{Z}$ is weakly subgroup separable, it follows that $\mathcal{S}(F_2,\mathbb{Z},\phi)$ is weakly subsemigroup separable by Lemma \ref{lem10}. 
	Since $F_2$ is finitely generated, $\mathcal{S}(F_2,\mathbb{Z},\phi)$ is finitely generated by Lemma $\ref{lem13}$. 
	It is clear that $\mathcal{S}(F_2,\mathbb{Z},\phi)$ is not commutative.  
	
	By Remark \ref{rem3} we have that $X_{\mathbb{Z}}$ is an infinite non-group $\mathcal{H}$-class. 
	Hence $\mathcal{S}(F_2,\mathbb{Z},\phi)$ is not strongly subsemigroup separable by Proposition \ref{lem4}. 
	Also the Sch\"utzenberger group of $X_{\mathbb{Z}}$ is isomorphic to $\mathbb{Z}$ and therefore is not weakly subsemigroup separable. 
	Notice that, due to the way the right and left actions of $F_2$ on $X_{\mathbb{Z}}$ are defined, the $\mathcal{H}$-class $X_{\mathbb{Z}}$ does not satisfy the condition of Lemma \ref{lem12}.
\end{eg}

We conclude this section by exhibiting an example of a finitely generated semigroup which is strongly subsemigroup separable but not completely separable. 

\begin{eg}
	\label{square-free}
	 Let $F_3=\{a,b,c\}^+$ be the free semigroup on the set $\{a, b,c\}$. 
	 Let $I \leq F_3$ be the ideal generated by the set $\{x^2 \mid x \in F_3\}$. 
	 Let $S=F_3/I$ be the Rees quotient of $F_3$ by $I$. 
	 We can view $S$ as the set of all square-free words over the alphabet $\{a,b,c\}$ with a zero adjoined. 
	 Multiplication in $S$ is concatenation, unless concatenation creates a word containing a subword which is a square, in which case the product is zero. 
	 Certainly $S$ is finitely generated by $\{a,b, c\}$. 
	
	First we will show that $S$ is not completely separable. 
	It is known that there exists an infinite square-free sequence $w=x_1x_2x_3\dots$ over $\{a,b,c\}$, see \cite[Chapter 2]{word}. 
	Then every finite prefix of $w$ is a non-zero element of $S$. 
	Let $w_i=x_1x_2\dots x_i \in S$. For $i < j$, let $v_{i,j}=x_{i+1}x_{i+2}\dots x_j \in S$. 
	Let $\sim$ be a finite index congruence class on $S$. 
	Then there exist $i, j \in \mathbb{N}$, with $i < j$, such that $w_i \sim w_j$. 
	Then \[w_j=w_iv_{i,j}\sim w_jv_{i,j}=w_iv_{i,j}v_{i,j}=0.\] 
	So we have shown that it is not possible for $0$ to be separated from $S \!\setminus\!\{0\}$ in a finite quotient. 
	Hence, $S$ is not completely separable.
	
	Now let $T \leq S$. 
	Then $0 \in T$. For $x \in S\!\setminus\! \{0\}$ let $|x|$ denote the length of $x$ in terms of the generators $\{a,b,c\}$. 
	Now let $v \notin T$ where $|v|=n$. 
	Let 
	\[I=\{x \in S \mid |x| > n\} \cup \{0\}.\] 
	Then $I$ is an ideal. 
	Clearly the Rees quotient $S/I$ is finite. 
	Furthermore, $[v]_{I}=\{v\}$. 
	Hence, $S$ is strongly subsemigroup separable. 
\end{eg}

\section{Semigroups with finitely many $\mathcal{H}$-classes}

In Section 3 we asked which of our separability properties are inherited by Sch\"utzenberger groups.
We showed in Corollaries \ref{cor2} and \ref{cor3} that the properties of complete separability and strong subsemigroup separability are inherited by Sch\"utzenberger groups. 
Although it is not true that every Sch\"utzenberger group of a weakly subsemigroup separable semigroup is itself weakly subsemigroup separable, we showed in Corollary \ref{lem1} that weak subsemigroup separability is inherited by Sch\"utzenberger groups of commutative semigroups.

One may ask whether the properties are inherited in the opposite direction, i.e., if every Sch\"utzenberger group of a semigroup $S$ has a separability property must $S$ itself satisfy the same property? 
This, however, turns out not to be true.
Let $\mathscr{P}$ be any of the properties of complete separability, strong subsemigroup separability, weak separability or residual finiteness. 
A semigroup whose Sch\"utzenberger groups all have property $\mathscr{P}$ may not itself have property $\mathscr{P}$. 
One example is the bicyclic monoid, given by the monoid presentation $\langle b,c \mid bc=1 \rangle$. 
The bicyclic monoid is $\mathcal{H}$-trivial, meaning that every $\mathcal{H}$-class is a singleton, so every Sch\"utzenberger group is the trivial group and certainly completely separable. 
However the bicyclic monoid is not even residually finite \cite[Corollary 1.12]{clifford1961algebraic}. 
In fact this direction fails comprehensively even for commutative semigroups, as the next example shows. 

\begin{eg}
	\label{eg2}
	\normalfont Let $A=\langle a \rangle \cong \mathbb{N}$. 
	Let $B=\{b_i \mid i \in \mathbb{N}\}\cup\{0\}$ be the countable null semigroup. 
	Let $S=A \cup B$ with multiplication between $A$ and $B$ as follows:
	
	\[a^ib_j=b_ja^i=\begin{cases} 
	b_{j-i} & \text{for } j > i, \\
	0 & \text{otherwise},
	\end{cases}\] 
	\[a^i0=0a^i=0.\]
	An exhaustive case analysis shows that this multiplication is associative and clearly it is commutative. 
	It is also straightforward to check that $S$ is $\mathcal{H}$-trivial. 
	However, $S$ is not residually finite.
	Suppose that $\sim$ is a finite index congruence on $S$. 
	Then there exist $i, j \in  \mathbb{N}$, with $i < j$, such that $b_i \sim b_j$. 
	Then 
	\[0=b_ia^{j-1} \sim b_ja^{j-1}=b_1.\] 
	So we cannot separate $0$ and $b_1$ in a finite quotient, and hence $S$ is not residually finite.
\end{eg}

\begin{rem}
	\normalfont In the semigroup $S$ of Example \ref{eg2}, both the ideal $B$ and the Rees quotient $S/B\cong\mathbb{N}_0$ are completely separable.
	 However, this is not enough to guarantee that $S$ is completely separable.
\end{rem}

In the remainder of this section we restrict our attention to the class of semigroups which have only finitely many $\mathcal{H}$-classes, or equivalently semigroups which have only finitely many left and right ideals. 
This is motivated by the following result.

\begin{thm}
\label{finiteresiduallyfinite}
\normalfont{\cite[Theorem 7.2]{GRAY201421}} \textit{Let $S$ be a semigroup with finitely many $\mathcal{H}$-classes. 
Then $S$ is residually finite if and only if all its Sch\"utzenberger groups are residually finite.} 
\end{thm}

We shall investigate whether there are analogous results for the properties of complete separability, strong subsemigroup separability and weak subsemigroup separability. 

For complete separability, the analogous result holds.

\begin{prop}
Let $S$ be a semigroup with only finitely many $\mathcal{H}$-classes. 
Then the following are equivalent:
\begin{enumerate}
	\item[\textup{(1)}] $S$ is completely separable;
	\item[\textup{(2)}] all the Sch\"utzenberger groups of $S$ are completely separable;
	\item[\textup{(3)}] $S$ is finite.
\end{enumerate}
\end{prop}

\begin{proof}
(1) $\Rightarrow$ (2). If $S$ is completely separable, then all of its Sch\"utzenberger groups are completely separable by Corollary \ref{cor3}.

(2) $\Rightarrow$ (3). If a Sch\"utzenberger group is completely separable it is finite by Lemma \ref{comp.sep.group}. 
As a Sch\"utzenberger group is in bijection with the corresponding $\mathcal{H}$-class, and $S$ has only finitely many $\mathcal{H}$-classes, we conclude that $S$ is finite. 

(3) $\Rightarrow$ (1). Clear.
\end{proof}

From Corollary \ref{cor2}, we know that every Schützenberger group of a strongly subsemigroup separable semigroup is itself strongly subsemigroup separable. 
However, even when a semigroup has only finitely many $\mathcal{H}$-classes, every Sch\"utzenberger group being strongly subsemigroup separable does not guarantee that the semigroup is strongly subsemigroup separable, as the following example demonstrates.

\begin{eg}
Let $G$ be an infinite strongly subsemigroup separable abelian group. 
The existence of such a group is established by Theorem \ref{thm3}. 
Then, recalling Construction \ref{con1}, $S=\mathcal{S}(G,G,\text{id})$ has three $\mathcal{H}$-classes: $G$, $X_G$ and $\{0\}$. 
The Sch\"utzenberger groups of the $\mathcal{H}$-classes are isomorphic to $G$, $G$ and the trivial group respectively. 
Then certainly every Sch\"utzenberger group is strongly subsemigroup separable. 
However, since $X_G$ is an infinite non-group $\mathcal{H}$-class, $S$ is not strongly subsemigroup separable by Proposition \ref{lem4}.
\end{eg}

The final property to consider is that of weak subsemigroup separability. 
Of all the separability properties considered in this paper, this is the only one which is not necessarily inherited by Sch\"utzenberger groups, as demonstrated by Example \ref{eg5}. 
However, when we restrict to a semigroup with only finitely many $\mathcal{H}$-classes, the following remains an open problem.

\begin{op}
\label{op}
Is it true that a semigroup with only finitely many $\mathcal{H}$-classes is weakly subsemigroup separable if and only if all its Sch\"utzenberger groups are weakly subsemigroup separable? 
\end{op}

Indeed, we do not even know if either direction of the above statement holds. 
In the rest of this section, we restrict our attention to locally finite semigroups with only finitely many $\mathcal{H}$-classes. 
By concentrating on this smaller class of semigroups we will be able to invoke Lemma \ref{locallyfinite}, which says that a semigroup which is both residually finite and locally finite is weakly subsemigroup separable. 
This line of investigation allows us to give the following partial answer to Open Problem \ref{op}.

\begin{thm}
\label{thm6}
Let $S$ be a semigroup with only finitely many $\mathcal{H}$-classes whose maximal subgroups are all solvable. 
Then $S$ is weakly subsemigroup separable if and only if all its Sch\"utzenberger groups are weakly subsemigroup separable.
\end{thm} 

In particular we note that a commutative semigroup with finitely many $\mathcal{H}$-classes is weakly subsemigroup separable if and only if all its Sch\"utzenberger groups are weakly subsemigroup separable. 
To prove Theorem \ref{thm6} we make use of several lemmas. 

A semigroup $S$ is called an \emph{epigroup} if every element of $S$ has a power which lies in a subgroup of $S$.
	
\begin{lem}
	\label{Lem1}
	A semigroup $S$ with finitely many $\mathcal{H}$-classes is an epigroup.
\end{lem}
	
\begin{proof}
	Let $s \in S$. 
	As $S$ has finitely many $\mathcal{H}$-classes there exist $i, \, j \in \mathbb{N}$ with $i < j$ such that $s^i \mathcal{H} s^j$.  
	Let $H$ be the $\mathcal{H}$-class of $s^i$. 
	Then $s^{j-i} \in \Stab(H)$. 
	Hence $(s^i)^{j-i}=(s^{j-i})^i \in\Stab(H)$, so $(s^i)^{j-i+1}=s^i(s^i)^{j-i}\in H$. 
	Therefore $H^{j-i+1}\cap H\neq\emptyset$, and so $H$ is a group. 
\end{proof}
	
\begin{lem}
\label{periodic}
	Let $S$ be a semigroup with finitely many $\mathcal{H}$-classes. 
	If every maximal subgroup of $S$ is torsion then $S$ is periodic.
\end{lem}
	
\begin{proof}
	By Lemma \ref{Lem1}, $S$ is an epigroup. 
	Let $s \in S$. 
	Then there is a power of $s$ in a torsion subgroup of $S$; in particular, there exists $i \in \mathbb{N}$ such that $s^i=e$, where $e$ is idempotent. 
	Hence $s^{2i}=e^2=e=s^i$ and $S$ is periodic. 
\end{proof}

\begin{cor}
\label{cor12}
Let $S$ be a semigroup with finitely many $\mathcal{H}$-classes. 
If $S$ is weakly subsemigroup separable then $S$ is periodic.
\end{cor}

\begin{proof}
As $S$ is weakly subsemigroup separable, then so are all of its maximal subgroups by Proposition $\ref{lem5}$.
Now if a group is not torsion then it contains a subgroup isomorphic to $\mathbb{Z}$ and therefore is not weakly subsemigroup separable by Proposition $\ref{lem5}$ and
Example \ref{eg1}. 
Thus all the maximal subgroups of S are torsion and the result follows by Lemma \ref{periodic}.
\end{proof}

\begin{cor}
\label{cor13}
Let $S$ be a semigroup with finitely many $\mathcal{H}$-classes. 
If every Sch\"utzenberger group of $S$ is weakly subsemigroup separable then $S$ is periodic. 
\end{cor}

\begin{proof}
If every Sch\"utzenberger group is weakly subsemigroup separable, then every maximal subgroup of $S$ is weakly subsemigroup separable and hence torsion, so $S$ is periodic by Lemma \ref{periodic}.
\end{proof}

At this point on our path to prove Theorem \ref{thm6}, we introduce Green's relation $\mathcal{J}$ on a semigroup $S$:
\[\mathcal{J}=\{(x,y) \mid S^1xS^1=S^1yS^1\}.\] 
This is an equivalence relation and $\mathcal{H} \subseteq \mathcal{J}$. 
It is also true that any ideal is a union of $\mathcal{J}$-classes. 
For more on Green's $\mathcal{J}$ relation see \cite[Chapter 2]{howie1995fundamentals}.

We show that an epigroup with finitely many $\mathcal{J}$-classes is locally finite if and only if all its maximal subgroups are locally finite. 
We first consider the case when a semigroup with a zero only has two $\mathcal{J}$-classes.

A semigroup $S$ with a zero is called \emph{0-simple} if $S^2 \neq \{0\}$ and the only ideals of $S$ are $\{0\}$ and $S$. 
A 0-simple semigroup is called \emph{completely 0-simple} if it is both 0-simple and an epigroup. 
For equivalent definitions of completely 0-simple semigroups see \cite[Theorem 3.2.11]{howie1995fundamentals}.
Rees showed that the class of completely 0-simple semigroups coincides with the class of \emph{Rees matrix semigroups over zero-groups} {\cite[Theorem 3.2.3]{howie1995fundamentals}}. 
For a group $G$, let the zero-group $G^0$ be $G$ with a zero adjoined. 
Let $I,\Lambda$ be non empty sets and $P=(p_{\lambda i})$ be a $\Lambda \times I$ matrix with entries from $G^0$ such that no row or column of $P$ consists entirely of zeros. 
The Rees matrix semigroup over the zero-group $G^{0}$ is $S=M^0[G;I,\Lambda;P]=(I \times G \times \Lambda) \cup \{0\}$ with multiplication given as follows:
\begin{equation*}
\begin{split}
(i,a,\lambda)(j,b,\mu)&=\begin{cases} (i,ap_{\lambda j},\mu) &\text{if } p_{\lambda i} \neq 0, \\ 0 &\text{if } p_{\lambda i} =0, \end{cases} \\
(i,a,\lambda)0&=0(i,a,\lambda)=0\cdot0=0.
\end{split}
\end{equation*}

In a completely 0-simple semigroup $M^0[G;I, \Lambda; P]$, the $\mathcal{H}$-class of the element $(i,a ,\lambda)$ is $\{i\} \times G \times \{\lambda\}$. 
This is a maximal subgroup if and only if $p_{\lambda i} \neq 0$, in which case it is isomorphic to $G$.
	
\begin{lem}
 \label{Lem3}
 A completely 0-simple semigroup $S=M^0[G;I, \Lambda; P]$ is locally finite if and only if $G$ is locally finite.
\end{lem}
	
\begin{proof}
	Suppose that $S$ is locally finite. 
	As $S$ has a subsemigroup isomorphic to $G$, and local finiteness is inherited by subsemigroups, it follows that $G$ is locally finite.
		
	Now suppose that $G$ is locally finite. 
	Let $T=\langle (i_1,g_1,\lambda_1)\, \dots, \, (i_n,g_n,\lambda_n) \rangle$. 
	We will show that the intersection of $T$ with a non-zero $\mathcal{H}$-class is finite. 
	Let $K$ be the subgroup of $G$ generated by the set 
	\[\{g_i \mid 1 \leq i \leq n\} \cup \{p_{\lambda_k i_l} \mid p_{\lambda_k i_l}\neq0, \,  1 \leq k, \, l\leq n\}.\] 
	As $K$ is finitely generated, it is finite. 
	Let $H=\{i\}\times G \times \{\lambda\}$ be a non-zero $\mathcal{H}$-class. 
	Then $H \cap T \subseteq \{i\} \times K \times \{\lambda\}$  and hence $H \cap T$ is finite. 
	As $T$ can only intersect finitely many non-zero $\mathcal{H}$-classes, it follows that $T$ is finite.
\end{proof}
	
\begin{lem}
	\label{localfiniteJ}
	Let $S$ be an epigroup with finitely many $\mathcal{J}$-classes. 
	Then $S$ is locally finite if and only if all its maximal subgroups are locally finite.
\end{lem}
	
\begin{proof}
		The forward direction follows as subsemigroups inherit local finiteness.
		
		Now assume that all the subgroups of $S$ are locally finite. 
		We will proceed by induction on the number of $\mathcal{J}$-classes.
		
		We may assume that $S$ has a zero. 
		If not then simply adjoin a zero. 
		Let $I$ be a 0-minimal ideal of $S$. 
		Then $I$ is either a null semigroup or it is $0$-simple by \cite[Proposition 2.4.9]{grillet1995semigroups}. 
		It is clear that any null semigroup is locally finite.
		Suppose then that $I$ is 0-simple.
		Clearly $I$ is an epigroup since $S$ is, so $I$ is completely $0$-simple.  
		The maximal subgroups of $I$ are the maximal subgroups of $S$ contained in $I$ by \cite[Proposition 2.4.2]{howie1995fundamentals}.
		Therefore, by Lemma \ref{Lem3} we have that $I$ is locally finite.
		If $I=S$ then we are done, so suppose that $I\neq S$. 
		
		Since the Rees quotient $S/I$ has one fewer $\mathcal{J}$-class than $S$, it is locally finite by the inductive hypothesis. 
		Let $T$ be a finitely generated subsemigroup of $S$. 
		As $S/I$ is locally finite, it follows that the Rees quotient $T/(T\cap I)$ is finite. 
		Hence $T \cap I$ is a subsemigroup with finite complement in $T$. 
		Therefore, as $T$ is finitely generated, $T \cap I$ is also finitely generated by \cite[Theorem 1.1]{10.1112/S0024611598000124}. 
		Then $T \cap I$ is finite as $I$ is locally finite. 
		Then $T=(T\!\setminus \!I) \cup (T \cap I)$ is finite and hence $S$ is locally finite.
\end{proof}

Lemmas \ref{Lem1} and \ref{localfiniteJ} together yield: 
	
\begin{cor}
\label{cor15}
	Let $S$ be a semigroup with finitely many $\mathcal{H}$-classes. 
	Then $S$ is locally finite if and only if all its maximal subgroups are locally finite.
\end{cor}

We are now in a position to prove Theorem $\ref{thm6}$.
	
\begin{proof}[Proof of Theorem \ref{thm6}]
Let $S$ be a semigroup with finitely many $\mathcal{H}$-classes whose maximal subgroups are all solvable. 

If $S$ is weakly subsemigroup separable then all its maximal subgroups are weakly subsemigroup separable by Proposition \ref{lem5}, and hence they are torsion by Proposition \ref{GroupSep}.  
Torsion solvable groups are locally finite, see \cite[5.4.11]{Robinson}.  
It follows from Corollary \ref{cor15} that $S$ is locally finite.  
Then certainly $S^1$ is locally finite.  
It is well known that the property of being locally finite is closed under subsemigroups and quotients.  
Therefore, each Sch\"utzenberger group, being a quotient of a subsemigroup of $S^1$, is locally finite. 
Since $S$ is residually finite, so are all it Sch\"utzenberger groups by Theorem \ref{finiteresiduallyfinite}.  
Hence, all the Sch\"utzenberger groups are weakly subsemigroup separable by Lemma \ref{locallyfinite}.

Now assume that all the Sch\"utzenberger groups of $S$ are weakly subsemigroup separable. 
Then they are certainly residually finite and it follows that $S$ is residually finite by Theorem \ref{finiteresiduallyfinite}. 
Furthermore, as $S$ only has finitely many $\mathcal{H}$-classes and all its Sch\"utzenberger groups are weakly subsemigroup separable, $S$ is periodic by Corollary \ref{cor13}. 
Then all its maximal subgroups are torsion and solvable so it follows that they are locally finite. 
Hence $S$ is locally finite by Corollary \ref{cor15} and therefore $S$ is weakly subsemigroup separable by Lemma \ref{locallyfinite}.  
\end{proof}
	
If there is any hope of solving Open Problem $\ref{op}$, we must consider cases where an infinite Sch\"utzenberger group is weakly subsemigroup separable but not solvable. 
The authors are aware of only a limited number of such groups. 
One such example is the Grigorchuk group, which is a finitely generated infinite torsion group that is weakly subgroup separable (and hence weakly subsemigroup separable), see \cite{grigorchuk_wilson_2003}. 
In particular, the following problem remains open.

\begin{op}
Let $G$ be the Grigorchuk group, let $I$ and $\Lambda$ finite sets, and let $P=(p_{\lambda i})$ a $\Lambda \times I$ matrix with entries from $G^0$ such that no row or column consists entirely of zeros. 
Is the semigroup $M^0[G;I,\Lambda; P]$ weakly subsemigroup separable?
\end{op}

\section{Acknowledgements}
The authors would like to thank the referee for their helpful comments and care with the paper.
The first author is grateful to EPSRC for financial support.
The second author is grateful to the School of Mathematics and Statistics of the University of St Andrews for financial support.

\end{document}